\documentclass[11pt,leqno]{article}

\usepackage[active]{srcltx}

\usepackage{amsfonts,latexsym,amsmath,amssymb,amsthm}
\usepackage{fullpage}
\newtheorem{theorem}{Theorem}[section]
\newtheorem{lemma}[theorem]{Lemma}

\newtheorem{corollary}[theorem]{Corollary}
\newtheorem{remark}[theorem]{Remark}

\theoremstyle{definition}

\theoremstyle{remark}

\newtheorem*{note*}{Note}

\numberwithin{equation}{section}

\makeatletter

\newcommand{\rank}{\mathop{\operator@font rank}}
\newcommand{\conv}{\mathop{\operator@font conv}}
\newcommand{\vol}{\mathop{\operator@font vol}}
\newcommand{\onetagright}{\tagsleft@false}
\makeatother

\newcommand{\ls}{\leqslant}
\newcommand{\gr}{\geqslant}
\renewcommand{\epsilon}{\varepsilon}

\usepackage{color}

\usepackage{hyperref}

\begin{document}
\small

\title{\bf Uniform cover inequalities for the volume of coordinate sections and projections of convex bodies}

\author{S.\ Brazitikos, A.\ Giannopoulos and D-M.\ Liakopoulos}

\date{}

\maketitle

\begin{abstract}
\footnotesize The classical Loomis-Whitney inequality and the uniform cover inequality of Bollob\'{a}s and Thomason
provide lower bounds for the volume of a compact set in terms of its lower dimensional coordinate projections.
We provide further extensions of these inequalities in the setting of convex bodies. We also establish the corresponding
dual inequalities for coordinate sections; these uniform cover inequalities for sections may be viewed as extensions of Meyer's
dual Loomis-Whitney inequality.
\end{abstract}

\section{Introduction}

The classical Loomis-Whitney inequality \cite{Loomis-Whitney-1949} compares the volume $|K|$ of a convex body $K$ in ${\mathbb R}^n$
with the geometric mean of the volumes $|P_i(K)|$ of its orthogonal projections onto $e_i^{\perp }$, where $\{e_1,\ldots ,e_n\}$
is an orthonormal basis of ${\mathbb R}^n$. We have
\begin{equation}\label{eq:intro-1}|K|^{n-1}\ls\prod_{i=1}^n|P_i(K)|\end{equation}
and equality holds if and only if $K$ is an orthogonal parallelepiped such that $\pm e_i$ are the normal vectors
of its facets. In this inequality, $|P_i(K)|$ denotes the $(n-1)$-dimensional volume of $P_i(K)$ (more generally,
when $A$ is a compact convex set in ${\mathbb R}^n$, we write $|A|$ for the volume of $A$ in the appropriate
affine subspace ${\rm aff}(A)$). In fact, \eqref{eq:intro-1} holds true for any compact subset $K$ of ${\mathbb R}^n$.

A dual inequality, in which the projections $P_i(K)$ are replaced by the sections $K\cap e^{\perp }$, was obtained
by Meyer in \cite{Meyer-1988}. For every convex body $K$ in ${\mathbb R}^n$ one has
\begin{equation}\label{eq:intro-2}|K|^{n-1}\gr \frac{n!}{n^n}\prod_{i=1}^n|K\cap e_i^{\perp }|\end{equation}
with equality if and only if $K$ is a linear image $T(B_1^n)$ of the cross-polytope $B_1^n={\rm conv}\{\pm e_1,\ldots ,\pm e_n\}$
for some diagonal (with respect to the given basis) operator $T={\rm diag}(\lambda_1,\ldots ,\lambda_n)$, $\lambda_i>0$.
Meyer's proof of this inequality is given for an unconditional convex body $K$, after observing that any Steiner symmetrization
of $K$ increases the right hand side of \eqref{eq:intro-2}.

Both inequalities have been generalized in the following setting: let $u_1,\ldots ,u_m$ be unit vectors in ${\mathbb R}^n$ and
let $c_1,\ldots ,c_m$ be positive real numbers such that John's condition
\begin{equation}\label{eq:intro-3}I_n=\sum_{i=1}^mc_iu_i\otimes u_i\end{equation}
is satisfied. Then, for every centered convex body $K$ in ${\mathbb R}^n$,
\begin{equation}\label{eq:intro-4}\frac{n!}{n^n}\prod_{i=1}^m|K\cap u_i^{\perp }|^{c_i}\ls |K|^{n-1}\ls\prod_{i=1}^m|P_{u_i^{\perp }}(K)|^{c_i}.\end{equation}
The assumption that $K$ is centered, i.e. it has its center of mass at the origin, is of course needed only
for the left hand side inequality. The equality cases are exactly the same with the ones in the Loomis-Whitney
and Meyer inequality respectively. The right hand side inequality in \eqref{eq:intro-4} was proved by Ball in \cite{Ball-1991}, while
the left hand side inequality was recently proved by Li and Huang in \cite{Li-Huang-2016}. The geometric Brascamp-Lieb inequality and its inverse,
due to Ball and Barthe (see \cite{Ball-handbook} and \cite{Barthe-1998}), play a crucial role in the proofs
of these more general inequalities.

A considerable extension of the Loomis-Whitney inequality was proved by Bollob\'{a}s and Thomason in \cite{Bollobas-Thomason-1995}.
In order to state their result, we introduce some notation and terminology. For every non-empty
$\tau\subset [n]:=\{1,\ldots ,n\}$ we set $F_{\tau }={\rm span}\{e_j:j\in\tau \}$ and $E_{\tau }=F_{\tau }^{\perp }$.
Given $s\gr 1$ and $\sigma\subseteq [n]$ we say that the (not necessarily distinct) sets $\sigma_1,\ldots ,\sigma_r\subseteq \sigma $
form an $s$-uniform cover of $\sigma $ if every $j\in \sigma $ belongs to exactly $s$ of the sets $\sigma_i$.
The {\it uniform cover inequality} of \cite{Bollobas-Thomason-1995} provides a lower bound for the volume of a compact
set in terms of the volumes of its coordinate projections that correspond to a uniform cover of $[n]$.

\begin{theorem}[Bollob\'{a}s-Thomason]\label{th:BT}Let $r\gr 1$ and let $(\sigma_1,\ldots ,\sigma_r)$ be an
$s$-uniform cover of $[n]$. For every compact subset $K$ of ${\mathbb R}^n$, which is the closure
of its interior, we have
\begin{equation}\label{eq:intro-5}|K|^s\ls\prod_{i=1}^r|P_{F_{\sigma_i}}(K)|.\end{equation}
\end{theorem}

In the first part of this article we obtain some {\it restricted} variants of the Loomis-Whitney inequality and of the
uniform cover inequality of Theorem \ref{th:BT}. Our starting point is the following inequality from
\cite{Giannopoulos-Hartzoulaki-Paouris-2002}: If $i\neq j\in \{ 1,\ldots ,n\}$
and $P_{ij}(K)=P_{E_{ij}}(K)$, where $E_{ij}={\rm span}\{e_i,e_j\}^{\perp }$ then
\begin{equation}\label{eq:intro-6}|P_i(K)|\, |P_j(K)|\gr \frac{n}{2(n-1)}\, |K|\, |P_{ij}(K)|.\end{equation}
This inequality may be viewed as a restricted (or ``local") version of the Loomis-Whitney inequality, in the sense that it gives a lower estimate for the
geometric mean of just two coordinate hyperplane projections of a convex body. A consequence of \eqref{eq:intro-6} is the inequality
\begin{equation}\label{eq:intro-7}\frac{S(P_{u^{\perp }}(K))}{|P_{u^{\perp }}(K)|}\ls \frac{2(n-1)}{n}\frac{S(K)}{|K|}\end{equation}
for every convex body $K$ in ${\mathbb R}^n$ and every $u\in S^{n-1}$, where $S(A)$ is the surface area of $A$ in the appropriate
dimension. This inequality was used in \cite{Giannopoulos-Hartzoulaki-Paouris-2002}
for the study of a question (posed by Dembo, Cover and Thomas \cite{Dembo-Cover-Thomas-1991}) about the monotonicity of an analogue of
the Fisher information on the class of compact convex sets, and it reappears in \cite{Giannopoulos-Koldobsky-Valettas-2016},
where the question to compare the surface area $S(K)$ of a convex body $K$ in ${\mathbb R}^n$ to the average, minimal or maximal
surface area of its  hyperplane projections is studied.

In Section 3 we revisit \eqref{eq:intro-6}. We adapt the proof of \cite[Lemma 4.1]{Giannopoulos-Hartzoulaki-Paouris-2002}
and combine it with the uniform cover inequality \eqref{eq:intro-5} of Theorem \ref{th:BT} to obtain the next generalization
of \eqref{eq:intro-6}.

\begin{theorem}\label{th:intro-1}Let $r>s\gr 1$, let $\sigma\subseteq [n]$ with cardinality $|\sigma |=d<n$ and let
$(\sigma_1,\ldots ,\sigma_r)$ be an $s$-uniform cover of $\sigma $. For every convex body $K$ in ${\mathbb R}^n$ we have
\begin{equation}\label{eq:intro-8}\prod_{i=1}^r|P_{E_{\sigma_i}}(K)|\gr \gamma (n,d,s,r)|P_{E_{\sigma }}(K)|^s\,|K|^{r-s},\end{equation}
where
\begin{equation}\label{eq:intro-9}\gamma (n,d,s,r)={n\choose d}^{r-s}{n-\frac{sd}{r}\choose n-d}^{-r}.\end{equation}
\end{theorem}

Note that if the sets $\sigma_1,\ldots ,\sigma_r$ have the same cardinality $k$, then $k=\frac{sd}{r}$ and the result takes the form
\begin{equation}\label{eq:intro-10}\prod_{i=1}^r|P_{E_{\sigma_i}}(K)|\gr {n\choose d}^{r-s}{n-k\choose n-d}^{-r}|P_{E_{\sigma }}(K)|^s\,|K|^{r-s}.\end{equation}
Our starting point \eqref{eq:intro-6} corresponds to the special case $d=r=2$, $k=1$ and $s=1$. The case $k=1$, $d=r$ and $s=1$ has been recently
studied by Soprunov and Zvavitch in \cite{Soprunov-Zvavitch-2016}; they use a similar argument, based on \cite[Lemma 4.1]{Giannopoulos-Hartzoulaki-Paouris-2002}
and on the classical Loomis-Whitney inequality. They also present an example which shows that the constant
\begin{equation}\gamma (n,r,1,r)={n\choose r}^{r-1}{n-1\choose n-r}^{-r}=\left (\frac{n}{r}\right )^r{n\choose r}^{-1}\end{equation}
is optimal.

\smallskip

In the second part of this article, starting from Meyer's inequality \eqref{eq:intro-2} we study the natural question if it is possible
to have an inequality for sections, which is dual to \eqref{eq:intro-6}. More precisely the question is if, for every centered convex body $K$ in ${\mathbb R}^n$
and every $i\neq j\in \{ 1,\ldots ,n\}$,
\begin{equation}\label{eq:intro-11}|K\cap e_i^{\perp }|\, |K\cap e_j^{\perp }|\ls c_0|K\cap E_{ij}|\,|K|,\end{equation}
where $c_0>0$ is an absolute constant. In Section 4 we exploit the main properties of the family of the $L_p$-centroid bodies $Z_p(K)$ of $K$
to show that this question has an affirmative answer. In a few words, through a duality argument, the question about coordinate sections of $K$
is translated to a question about coordinate projections of some projection of a suitable centroid body of $K$, and then one may use
the Loomis-Whitney inequality (or some extension of it) to complete the proof. Generalizing the method and making full use of the uniform cover inequality
of Bollob\'{a}s and Thomason, one can prove more general inequalities of this form, in the spirit of Theorem \ref{th:intro-1}.

\begin{theorem}\label{th:intro-2}Let $r>s\gr 1$, let $\sigma\subseteq [n]$ with cardinality $|\sigma |=d<n$ and let
$(\sigma_1,\ldots ,\sigma_r)$ be an $s$-uniform cover of $\sigma $. Let also $d_i=|\sigma_i|$. For every centered convex body $K$ in ${\mathbb R}^n$ we have
\begin{equation}\label{eq:intro-12}\prod_{i=1}^r|K\cap E_{\sigma_i}|\ls \frac{(c_0d)^{ds}}{d_1^{d_1}\cdots d_r^{d_r}}|K\cap E_{\sigma }|^s|K|^{r-s},\end{equation}
where $c_0>0$ is an absolute constant.
\end{theorem}

Note that under the assumptions of Theorem \ref{th:intro-2} we have $d_1+\cdots +d_r=ds$, and hence 
$$d_1^{d_1}\ldots d_r^{d_r}\gr\left(\frac{ds}{r}\right)^{ds}$$
by Jensen's inequality. Therefore, the result may be written in the simpler form
\begin{equation}\label{eq:intro-13}\prod_{i=1}^r|K\cap E_{\sigma_i}|\ls \left (\frac{c_0r}{s}\right )^{ds}|K\cap E_{\sigma }|^s|K|^{r-s}.\end{equation}
This is equivalent to \eqref{eq:intro-12} if all the sets $\sigma_i$ have the same cardinality $k=\frac{ds}{r}$. Our starting point \eqref{eq:intro-11} corresponds to the special case $d=r=2$, $k=1$ and $s=1$. In the more general case $d=r$, $k=1$ and $s=1$, which corresponds to $\sigma_j=\{i_j\}$ for some distinct $i_1,\ldots ,i_r\in [n]$, Theorem
\ref{th:intro-2} provides the bound
\begin{equation}\label{eq:intro-13b}\prod_{j=1}^r|K\cap e_{i_j}^{\perp }|\ls \left (c_0r\right )^r|K\cap [{\rm span}\{e_{i_1},\ldots ,e_{i_r}\}]^{\perp }|\,|K|^{r-1}.\end{equation}
The constant $(c_0r)^r$ is probably non optimal but it depends only on $r$ and not on the dimension $n$.

\smallskip

In Section 5 we provide an alternative proof of \eqref{eq:intro-6}, with the same constant, using
a general inequality about mixed volumes. Let ${\cal C}=(K_3,\ldots ,K_n)$ be an $(n-2)$-tuple of compact convex sets
in ${\mathbb R}^n$. For any pair of compact convex sets $A,B$ in ${\mathbb R}^n$ we denote the mixed volume
$V(A,B,{\cal C})$ by $V(A,B)$ (see Section 2 for basic facts about mixed volumes). Then, for any triple $A,B,C$
of compact convex sets in ${\mathbb R}^n$ we have
\begin{equation}\label{eq:intro-14}V(A,A)V(B,C)\ls 2V(A,B)V(A,C).\end{equation}
In fact, \eqref{eq:intro-14} is an immediate consequence of one of the main lemmas in \cite{Giannopoulos-Hartzoulaki-Paouris-2002}
and \cite{Fradelizi-Giannopoulos-Meyer-2003}. We observe that \eqref{eq:intro-14} leads to a generalization of
\eqref{eq:intro-6}, valid for any pair of hyperplane projections defined by two not necessarily orthogonal unit vectors $u$ and $v$.

\begin{theorem}\label{th:intro-3}
Let $K$ be a convex body in ${\mathbb R}^n$ and $u,v\in S^{n-1}$.
If $P_{u,v}(K)=P_{{\rm span}\{u,v\}^{\perp }}(K)$, then
\begin{equation}\label{eq:intro-15}|P_u(K)|\, |P_v(K)|\gr \frac{n}{2(n-1)}\sqrt{1-\langle u,v\rangle ^2} \,|K|\, |P_{u,v}(K)|.\end{equation}
\end{theorem}

We also discuss a different question, which illustrates the usefulness of \eqref{eq:intro-14}. It has been conjectured by Hug and Schneider in \cite{Hug-Schneider-2011}
that for any $1\ls r\ls n$ and any $r$-tuple $(K_1,\ldots ,K_r)$ of convex bodies in ${\mathbb R}^n$ one has
\begin{equation}\label{eq:intro-16}V(K_1,\ldots ,K_r,B_2^n[n-r])\ls \frac{(n-r)!\omega_{n-r}}{n!}\prod_{i=1}^rV_1(K_i),\end{equation}
where $V(A_1,\ldots ,A_n)$ is the mixed volume of $n$ compact convex sets $A_i$, the notation $A[m]$ stands for an $m$-tuple $A,\ldots ,A$, and
\begin{equation}\omega_{n-s}V_s(K)={n\choose s}V(K[s],B_2^n[n-s])\end{equation}
is the $s$-th intrinsic volume of $K$ (see also \cite{Betke-Weil-1991}
for the planar case). Hug and Schneider proved \eqref{eq:intro-16} under the assumption that the bodies $K_1,\ldots ,K_r$ are zonoids.
In the case $r=2$, Artstein-Avidan, Florentin and Ostrover have proved in \cite{Artstein-Florentin-Ostrover-2014}
that if $K$ is any convex body and $Z$ is a zonoid in ${\mathbb R}^n$ then
\begin{equation}\label{eq:intro-17}|B_2^n|\,V(K,Z,B_2^n[n-2])\ls \frac{n}{n-1}\frac{\omega_n\omega_{n-2}}{\omega_{n-1}^2}\, V(K,B_2^n[n-1])\,V(Z,B_2^n[n-1]).
\end{equation}
By the definition of $V_1(K)$ this inequality is the same as the conjectured one (for $r=2$).

A discussion of a more general problem is given in \cite{Soprunov-Zvavitch-2016}, where Soprunov and Zvavitch prove that if $A$ is any
convex body in ${\mathbb R}^n$ and $Z_1,\ldots ,Z_r$ are zonoids then
\begin{equation}\label{eq:intro-19}|A|^{r-1}V(Z_1,\ldots ,Z_r,A[n-r])\ls r^{r-1}\prod_{i=1}^rV(Z_i,A[n-1]),\end{equation}
while for an $r$-tuple of (arbitrary) convex bodies $K_1,\ldots ,K_r$ in ${\mathbb R}^n$ one has
\begin{equation}\label{eq:intro-20}|A|^{r-1}V(K_1,\ldots ,K_r,A[n-r])\ls c_{n,r}\prod_{i=1}^rV(K_i,A[n-1]),\end{equation}
where $c_{n,r}=n^rr^{r-1}$. Moreover, the constant $c_{n,r}$ can be replaced by $c_{n,r}^{\prime }=n^{r/2}r^{r-1}$ if $K_1,\ldots ,K_r$ are origin symmetric.

We observe that \eqref{eq:intro-14} implies a much more general inequality, which confirms the conjectured
inequality \eqref{eq:intro-16} in the case $r=2$, with an absolute (almost optimal) constant and shows that the constant $c_{n,2}$ in \eqref{eq:intro-20}
may be replaced by the constant $2$.

\begin{theorem}\label{th:intro-4}Let $A$ be a convex body in ${\mathbb R}^n$. Then, for any pair of convex bodies $K_1$ and $K_2$ in ${\mathbb R}^n$,
\begin{equation}\label{eq:intro-18} |A|\,V(K_1,K_2,A[n-2])\ls 2V(K_1,A[n-1])\,V(K_2,A[n-1]).
\end{equation}
\end{theorem}

\smallskip

Choosing $A=B_2^n$ in Theorem \ref{th:intro-4} we get a variant of \eqref{eq:intro-16} with constant $2$. One can check that $\frac{n-1}{n}<\frac{\omega_n\omega_{n-2}}{\omega_{n-1}^2}<1$, and hence the conjectured constant
$b_{n,2}:=\frac{n}{n-1}\frac{\omega_n\omega_{n-2}}{\omega_{n-1}^2}$ satisfies
$$1<b_{n,2}<\frac{n}{n-1}.$$
In other words, the constant in Theorem \ref{th:intro-4} is worse than the conjectured one (only) by a factor $2$.

Regarding the constants $c_{n,r}$ and $c_{n,r}^{\prime }$ in \eqref{eq:intro-20}, from Theorem \ref{th:intro-4} we immediately
see that $c_{n,2}\ls 2$ and we also observe that an induction argument leads to a version of the general inequality
\eqref{eq:intro-20} with a constant $c_r$ which depends only on $r$. It would be interesting to determine the best possible
value of this constant; simple induction gives the very crude estimate $c_r\ls 2^{2^{r-1}-1}$.

\section{Notation and background information}

We work in ${\mathbb R}^n$, which is equipped with a Euclidean structure $\langle \cdot,\cdot\rangle$
and we fix an orthonormal basis $\{e_1,\ldots ,e_n\}$. We denote by $B_2^n$ and $S^{n-1}$ the Euclidean unit ball and sphere
in ${\mathbb R}^n$ respectively. We write $\sigma $ for the normalized rotationally invariant probability measure on $S^{n-1}$ and $\nu $
for the Haar probability measure on the orthogonal group $O(n)$. Let $G_{n,k}$ denote the Grassmannian of all $k$-dimensional
subspaces of ${\mathbb R}^n$. Then, $O(n)$ equips $G_{n,k}$ with a Haar probability measure $\nu_{n,k}$.
The letters $c,c^{\prime }, c_1, c_2$ etc. denote absolute positive constants which may change from line to line. Whenever we write
$a\simeq b$, we mean that there exist absolute constants $c_1,c_2>0$ such that $c_1a\ls b\ls c_2a$.

Let ${\mathcal K}_n$ denote the class of all non-empty compact convex subsets of ${\mathbb R}^n$.
If $K\in {\mathcal K}_n$ has non-empty interior, we will say that $K$ is a convex body. If $A\in {\mathcal K}_n$,
we will denote by $|A|$ the volume of $A$ in the appropriate affine subspace unless otherwise stated. The volume of $B_2^n$ is denoted by $\omega_n$.
We say that a convex body $K$ in ${\mathbb R}^n$ is symmetric if $x\in K$ implies that $-x\in K$, and that $K$ is centered
if its center of mass $\frac{1}{|K|}\int_Kx\,dx $ is at the origin. The support function of a
convex body $K$ is defined by $h_K(y)=\max \{\langle x,y\rangle :x\in K\}$, and the mean width of $K$ is
\begin{equation}\label{eq:not-1}w(K)=\int_{S^{n-1}}h_K(\theta )\,d\sigma (\theta ). \end{equation}
For any $E\in G_{n,k}$ we denote by $E^{\perp }$ the orthogonal subspace of $E$, i.e. $E^{\perp }=\{x\in {\mathbb R}^n:
\langle x,y\rangle =0\;\hbox{for all}\,y\in E\}$. In particular, for any $u\in S^{n-1}$ we define $u^{\perp }=\{x\in {\mathbb R}^n:\langle x,u\rangle =0\}$.
The section of $K\in {\mathcal K}_n$ with a subspace $E$ of ${\mathbb R}^n$ is $K\cap E$, and the orthogonal projection of $K$ onto $E$ is denoted
by $P_E(K)$.

Mixed volumes are introduced by a classical theorem of Minkowski which describes the way volume behaves with respect
to the operations of addition and multiplication of compact vonvex sets by non-negative reals: If $K_1,\ldots ,K_N\in {\cal K}_n$, $N\in {\mathbb N}$,
then the volume of $t_1K_1+\cdots +t_NK_N$ is a homogeneous polynomial of degree $n$ in $t_i\gr 0$ (see \cite{Burago-Zalgaller-book} and \cite{Schneider-book}):
\begin{equation}\label{eq:not-2}\big |t_1K_1+\cdots +t_NK_N\big |=\sum_{1\ls i_1,\ldots ,i_n\ls N} V(K_{i_1},\ldots
,K_{i_n})t_{i_1}\ldots t_{i_n},\end{equation}
where the coefficients $V(K_{i_1},\ldots ,K_{i_n})$ are chosen to be invariant under permutations of their
arguments. The coefficient $V(K_{i_1},\ldots ,K_{i_n})$ is called the mixed volume of the $n$-tuple $(K_{i_1},\ldots ,K_{i_n})$.
We will often use the fact that $V$ is positive linear with respect to each of its arguments and that
$V(K,\ldots ,K)=|K|_n$ (the $n$-dimensional Lebesgue measure of $K$) for all $K\in {\cal K}_n$.

Steiner's formula is a special case of Minkowski's theorem. The volume of $K+tB_2^n$, $t>0$, can be expanded as a polynomial in $t$:
\begin{equation}\label{eq:not-3}|K+tB_2^n|=\sum_{k=0}^n{n\choose k}W_k(K)t^k,\end{equation}
\noindent where $W_k(K):=V(K[n-k],B_2^n[k])$ is the $k$-th quermassintegral of $K$.

The Aleksandrov-Fenchel inequality states that if $K,L,K_3,\ldots ,K_n\in {\cal K}_n$, then
\begin{equation}\label{eq:not-4}V(K,L,K_3,\ldots ,K_n)^2\gr V(K,K,K_3,\ldots ,K_n) V(L,L,K_3,\ldots ,K_n).\end{equation}
In particular, this implies that the sequence $(W_0(K),\ldots ,W_n(K))$ is
log-concave. From the Aleksandrov-Fenchel inequality one can
recover the Brunn-Minkowski inequality as well as the following generalization
for the quermassintegrals:
\begin{equation}\label{eq:not-6}W_k(K+L)^{\frac{1}{n-k}}\gr W_k(K)^{\frac{1}{n-k}}+W_k(L)^{\frac{1}{n-k}},\qquad k=0,\ldots ,n-1.\end{equation}
We write $S(K)$ for the surface area of $K$. From Steiner's formula and the definition of surface area we see that $S(K)=nW_1(K)$.
Finally, let us mention Kubota's integral formula
\begin{equation}\label{eq:not-7}W_k(K)=\frac{\omega_n}{\omega_{n-k}}\int_{G_{n,n-k}}
|P_E(K)|\,d\nu_{n,n-k}(E),\qquad 1\ls k\ls n-1.\end{equation}
The case $k=1$ is Cauchy's surface area formula
\begin{equation}\label{eq:not-8}S(K)=\frac{\omega_n}{n\omega_{n-1}}\int_{S^{n-1}}|P_{u^{\perp }}(K)|\,d\sigma (u).\end{equation}
We refer to the books \cite{Gardner-book} and \cite{Schneider-book} for basic facts from the Brunn-Minkowski theory and to the book
\cite{AGA-book} for basic facts from asymptotic convex geometry. We also refer to \cite{BGVV-book} for
detailed information on the properties of the family of the $L_p$-centroid bodies of a convex body.

\section{Restricted Loomis-Whitney inequalities}

For the proof of Theorem \ref{th:intro-1}, we will use the uniform cover inequality \eqref{eq:intro-5} of Bollob\'{a}s and Thomason and the
next classical inequality of Berwald \cite{Berwald-1947}.

\begin{lemma}\label{lem:berwald}
Let $A$ be a convex body in ${\mathbb R}^m$ and let $\phi :A\rightarrow {\mathbb R}^+$ be a
concave function. Then, for every $0<p<q$,
\begin{equation}\label{eq:berwald-1}\left [{m+q\choose m}\frac{1}{|A|}\int_A|\phi (x)|^qdx\right ]^{1/q}\ls
\left [{m+p\choose m}\frac{1}{|A|}\int_A|\phi (x)|^pdx\right ]^{1/p}.\end{equation}
\end{lemma}

\medskip

\noindent\textbf{Proof of Theorem \ref{th:intro-1}.} Let $r>s\gr 1$, let $\sigma\subseteq [n]$ with cardinality $|\sigma |=d<n$ and let
$(\sigma_1,\ldots ,\sigma_r)$ be an $s$-uniform cover of $\sigma $. Note that if $|\sigma_i|=d_i$ then
$$ds=d_1+\cdots +d_r.$$
For every $y\in P_{E_{\sigma }}(K)$ we define the sets
\begin{equation}\label{eq:berwald-2}K_{i}(y)=\left\{ t\in F_{\sigma\setminus\sigma_i}:y+t\in P_{E_{\sigma_i}}(K)\right\}\end{equation}
and
\begin{equation}\label{eq:berwald-3}K(y)=\{ t\in F_{\sigma }:y+t\in K\}.\end{equation}
Then, $K_i(y)$ is the orthogonal projection of $K(y)$ onto $F_{\sigma\setminus\sigma_i}$. Since
$(\sigma_1,\ldots ,\sigma_r)$ is an $s$-uniform cover of $\sigma $, we have that $(\sigma\setminus\sigma_1,\ldots ,\sigma\setminus\sigma_r)$
is an $(r-s)$-uniform cover of $\sigma $. It follows from \eqref{eq:intro-5} that
\begin{equation}\label{eq:berwald-4}|K(y)|^{r-s}\ls \prod_{i=1}^r|K_i(y)|\end{equation}
for every $y\in P_{E_{\sigma }}(K)$. An application of H\"{o}lder's inequality shows that
\begin{align}\label{eq:berwald-5}
\prod_{i=1}^r|P_{E_{\sigma_i }}(K)| &= \prod_{i=1}^r \int_{P_{E_{\sigma }}(K)}
|K_i(y)|dy\gr \left (\int_{P_{E_{\sigma }}(K)}(|K_1(y)|\cdots |K_r(y)|)^{1/r}dy\right )^r\\
\nonumber &\gr \left (\int_{P_{E_{\sigma }}(K)}|K(y)|^{\frac{r-s}{r}}dy\right )^r.
\end{align}
By the Brunn-Minkowski inequality, the function $\phi :P_{E_{\sigma }}(K)\rightarrow {\mathbb R}$ defined by
$\phi (y)=|K(y)|^{1/d}$ is concave, and
$$|K(y)|^{\frac{r-s}{r}}=\phi(y)^{\frac{(r-s)d}{r}}=\phi (y)^{d-\frac{d_1+\cdots +d_r}{r}}.$$
Note that
\begin{equation}\int_{P_{E_{\sigma }}(K)}\phi (y)^d\,dy =\int_{P_{E_{\sigma }}(K)}|K(y)|\,dy =|K|.\end{equation}
Applying Lemma \ref{lem:berwald} with $A=P_{E_{\sigma }}(K)$,
$m=n-d$, $p=\frac{(r-s)d}{r}$ and $q=d$, we get
\begin{align}\label{eq:berwald-6}
\left [{n-d+\frac{(r-s)d}{r}\choose n-d}\frac{1}{|P_{E_{\sigma }}(K)|}\int_{P_{E_{\sigma }}(K)}|K(y)|^{\frac{r-s}{r}}dy\right ]^r
&=  \left [{n-\frac{sd}{r}\choose n-d}\frac{1}{|P_{E_{\sigma }}(K)|}\int_{P_{E_{\sigma }}(K)}\phi (y)^{\frac{(r-s)d}{r}}\,dy\right ]^r\\
\nonumber &\gr  \left [{n\choose d}\frac{1}{|P_{E_{\sigma }}(K)|}\int_{P_{E_{\sigma }}(K)}\phi (y)^d\,dy\right ]^{r-s}\\
\nonumber &= \left [{n\choose d}\frac{1}{|P_{E_{\sigma }}(K)|}\,|K|\right ]^{r-s}.
\end{align}
It follows that
\begin{equation}\label{eq:berwald-7}\left (\int_{P_{E_{\sigma }}(K)}|K(y)|^{\frac{r-s}{r}}dy\right )^r\gr {n\choose d}^{r-s}{n-\frac{sd}{r}\choose n-d}^{-r}|P_{E_{\sigma }}(K)|^{s}\,|K|^{r-s},\end{equation}
and the result follows from \eqref{eq:berwald-5}. $\hfill\Box $

\begin{remark}\rm If the sets $\sigma_1,\ldots ,\sigma_r$ have the same cardinality $k$ then $k=\frac{sd}{r}$ and the result takes the form
\begin{equation}\label{eq:berwald-8}\prod_{i=1}^r|P_{E_{\sigma_i }}(K)|\gr {n\choose d}^{r-s}{n-k\choose n-d}^{-r}|P_{E_{\sigma }}(K)|^{s}\,|K|^{r-s}.\end{equation}
In order to get a feeling of the estimates, let us consider the case of two orthogonal coordinate subspaces $F_1,F_2\in G_{n,k}$, where $k<n/2$.
Then, $r=2$, $s=1$ and $d=2k$. Therefore,
\begin{equation}\gamma (n,2k,1,2)={n\choose 2k}{n-k\choose k}^{-2}\gr c_1^k\end{equation}
for some absolute constant $c_1>0$. So, we get:
\end{remark}

\begin{corollary}\label{cor:intro-1}Let $k<n/2$ and let $F_1,F_2\in G_{n,k}$ be two orthogonal coordinate subspaces.
For every convex body $K$ in ${\mathbb R}^n$ we have
\begin{equation}\label{eq:berwald-9}|P_{F_1^{\perp }\cap F_2^{\perp }}(K)|\,|K|\ls c^k|P_{F_1^{\perp }}(K)||P_{F_2^{\perp }}(K)|,\end{equation}
where $c>0$ is an absolute constant.
\end{corollary}

\section{Restricted dual Loomis-Whitney inequalities}

Let $K$ be a centered convex body of volume $1$ in ${\mathbb R}^n$. Recall that, for
every $p\gr  1$, the $L_p$-centroid body $Z_p(K)$ of $K$ is the symmetric convex body with support function
\begin{equation}\label{eq:sections-1}h_{Z_p(K)}(y)=\|\langle\cdot ,y\rangle\|_{L^p(K)}
=\left (\int_K|\langle x,y\rangle |^pdx\right )^{1/p}.
\end{equation}
The $L_p$-centroid bodies of a convex body were introduced by Lutwak and Zhang. Their systematic study from an asymptotic
point of view started with the works of Paouris \cite{Paouris-2006} and \cite{Paouris-2012}. In particular, the inequality
\eqref{eq:sections-3} below, which is essential for our argument, comes from \cite{Paouris-2012}.
We will use the next basic facts about the family $\{ Z_p(K)\}_{p\gr 1}$; see \cite[Chapter 5]{BGVV-book} for the proofs.

\begin{lemma}[$L_p$-centroid bodies]\label{lem:BGVV}There exist absolute constants $c_i>0$ such that, for every centered convex body $K$ of volume $1$ in ${\mathbb R}^n$,
for every $1\ls k\ls n-1$, $q>p\gr 1$ and $F\in G_{n,k}$, we have
\begin{equation}\label{eq:sections-2}Z_q(K)\subseteq \frac{c_1q}{p}Z_p(K)\end{equation}
and
\begin{equation}\label{eq:sections-3}
c_2\ls |K\cap F^{\perp}|^{\frac{1}{k}}|P_{F}(Z_k(K))|^{\frac{1}{k}}
\ls c_3.\end{equation}
Moreover, if $p\gr n$ then we have that
\begin{equation}\label{eq:sections-4}Z_p(K)\supseteq c_4Z_{\infty }(K),\end{equation}
where $Z_{\infty }(K)={\rm conv}\{K,-K\}$.\end{lemma}

Besides \eqref{eq:sections-2} and \eqref{eq:sections-3} we will need the following: For every centered convex body $K$ of volume $1$ in ${\mathbb R}^n$,
for every $p\gr 1$ and every $u\in S^{n-1}$.
\begin{equation}\label{eq:sections-5}c_5h_{Z_p(K)}(u)\ls \frac{1}{|K\cap u^{\perp }|}\ls c_6ph_{Z_p(K)}(u),\end{equation}
where $c_5,c_6>0$ are absolute constants.

\smallskip

We start with the proof of \eqref{eq:intro-6}. This is a simple case of the general inequality of Theorem \ref{th:intro-2}, which
illustrates the main ideas behind its proof.

\begin{theorem}\label{th:sections}Let $K$ be a centered convex body in ${\mathbb R}^n$ and let $u,v$ be orthogonal unit vectors in ${\mathbb R}^n$.
If $E_{uv}=[{\rm span}\{u,v\}]^{\perp }$ then
\begin{equation}\label{eq:sections-6}|K\cap u^{\perp }|\,|K\cap v^{\perp }|\ls c|K\cap E_{uv}|\,|K|,\end{equation}
where $c>0$ is an absolute constant.
\end{theorem}

\begin{proof}By homogeneity we may assume that $|K|=1$. Using \eqref{eq:sections-3} with $F=E_{uv}^{\perp }={\rm span}\{u,v\}$ we see that
\begin{equation}\label{eq:sections-7}|K\cap E_{uv}|\gr \frac{c_7}{|P_F(Z_2(K))|}.\end{equation}
From \eqref{eq:sections-4} we have
\begin{equation}\label{eq:sections-8}|K\cap u^{\perp }|\,|K\cap v^{\perp }|\ls c_8\,[h_{Z_1(K)}(u)h_{Z_1(K)}(v)]^{-1}.\end{equation}
From \eqref{eq:sections-2} we also have
\begin{equation}\label{eq:sections-9}h_{Z_1(K)}(u)\gr c_{10}h_{Z_2(K)}(u)=c_{10}h_{P_F(Z_2(K))}(u)\quad\hbox{and}\quad
h_{Z_1(K)}(v)\gr c_{10}h_{Z_2(K)}(v)=c_{10}h_{P_F(Z_2(K))}(v),\end{equation}
where the two equalities hold because $u,v\in F$. If we consider the two-dimensional origin symmetric ellipsoid $C=P_F(Z_2(K))$ it is clear
(by the Loomis-Whitney inequality in the plane) that
\begin{equation}\label{eq:sections-10}|C|\ls 4h_C(u)h_C(v),\end{equation}
and this shows that
\begin{align}\label{eq:sections-11}c_7|K\cap E_{uv}|^{-1} &\ls |P_F(Z_2(K))|\ls 4h_{P_F(Z_2(K))}(u)h_{P_F(Z_2(K))}(v)\ls 4c_{10}^{-2}h_{Z_1(K)}(u)h_{Z_1(K)}(v)\\
\nonumber &\ls 4c_8c_{10}^{-2}\big (|K\cap u^{\perp }|\,|K\cap v^{\perp }|\big )^{-1}.
\end{align}
This completes the proof. \end{proof}

\medskip

\noindent {\bf Proof of Theorem \ref{th:intro-2}.} Let $r>s\gr 1$, let $\sigma\subseteq [n]$ with cardinality $|\sigma |=d<n$ and let
$(\sigma_1,\ldots ,\sigma_r)$ be an $s$-uniform cover of $\sigma $. Note that if $|\sigma_i|=d_i$ then $ds=d_1+\cdots +d_r$.

By homogeneity we may assume that $|K|=1$. Starting from \eqref{eq:sections-3} we may write
\begin{equation}\label{eq:sections-12}c_2\ls |K\cap E_{\sigma_i}|^{\frac{1}{d_i}}|P_{F_{\sigma_i}}(Z_{d_i}(K))|^{\frac{1}{d_i}}\ls c_3\end{equation}
for all $i$, and hence,
\begin{equation}\label{eq:sections-13}\prod_{i=1}^r|K\cap E_{\sigma_i}|\ls c_3^{d_1+\cdots +d_r}\prod_{i=1}^r|P_{F_{\sigma_i}}(Z_{d_i}(K))|^{-1}
=c_3^{ds}\prod_{i=1}^r|P_{F_{\sigma_i}}(Z_{d_i}(K))|^{-1},\end{equation}
so we need a lower bound for the product
\begin{equation}\label{eq:sections-14}\prod_{i=1}^r|P_{F_{\sigma_i}}(Z_{d_i}(K))|.\end{equation}
From \eqref{eq:sections-2} we have
\begin{equation}\label{eq:sections-15}Z_d(K)\subseteq \frac{c_1d}{d_i}\,Z_{d_i}(K)\end{equation}
for all $i=1,\ldots ,r$, which gives
\begin{equation}\label{eq:sections-16}\prod_{i=1}^r|P_{F_{\sigma_i}}(Z_{d}(K))|\ls \prod_{i=1}^r\left (\frac{c_1d}{d_i}\right )^{d_i}\prod_{i=1}^r|P_{F_{\sigma_i}}(Z_{d_i}(K))|= \frac{(c_1d)^{ds}}{d_1^{d_1}\cdots d_r^{d_r}}\prod_{i=1}^r|P_{F_{\sigma_i}}(Z_{d_i}(K))|.\end{equation}
Now, since $(\sigma_1,\ldots ,\sigma_r)$ is an $s$-uniform cover of $\sigma $, applying the uniform cover inequality of Bollob\'{a}s
and Thomason to the convex body $P_{F_{\sigma }}(Z_d(K))$ we get
\begin{equation}\label{eq:sections-17}|P_{F_{\sigma }}(Z_d(K))|^s\ls \prod_{i=1}^r|P_{F_{\sigma_i}}(Z_d(K))|.\end{equation}
Next, using again \eqref{eq:sections-3}, we see that
\begin{equation}\label{eq:sections-18}|P_{F_{\sigma }}(Z_d(K))|^s\gr c_2^{ds}|K\cap E_{\sigma }|^{-s}.\end{equation}
Combining the above, we have
\begin{equation}\label{eq:sections-19}\prod_{i=1}^r|K\cap E_{\sigma_i}|\ls \frac{(c_1c_3d)^{ds}}{d_1^{d_1}\cdots d_r^{d_r}}|P_{F_{\sigma }}(Z_d(K))|^{-s}
\ls \frac{(c_0d)^{ds}}{d_1^{d_1}\cdots d_r^{d_r}}|K\cap E_{\sigma }|^s,\end{equation}
where $c_0=c_1c_3/c_2$, and the result follows. $\hfill\Box $

\medskip

In order to get a feeling of the estimates, let us consider the case of two orthogonal coordinate subspaces $F_1,F_2\in G_{n,k}$, where $k<n/2$.
Then, $r=2$, $s=1$ and $d=2k$. Therefore, we get:

\begin{corollary}\label{cor:intro-2}Let $k<n/2$ and let $F_1,F_2\in G_{n,k}$ be two orthogonal coordinate subspaces.
For every centered convex body $K$ in ${\mathbb R}^n$ we have
\begin{equation}\label{eq:sections-20}|K\cap F_1^{\perp }|\,|K\cap F_2^{\perp }|\ls c^k|K\cap F_1^{\perp }\cap F_2^{\perp }|\,|K|,\end{equation}
where $c>0$ is an absolute constant.
\end{corollary}

\section{Inequalities about mixed volumes}

In this last section we prove Theorem \ref{th:intro-3} and we discuss the conjecture of Hug and Schneider in the case $r=2$; we provide
an affirmative answer, up to a factor $2$, in greater generality. The main source of our results is the next lemma which is an almost
immediate consequence of a lemma from \cite{Fradelizi-Giannopoulos-Meyer-2003} (a variant of it had been earlier proved
in \cite{Giannopoulos-Hartzoulaki-Paouris-2002}). We reproduce a sketch of its proof for completeness.

\begin{lemma}\label{lem:AF-1}
Let ${\cal C}=(K_3,\ldots ,K_n)$ be an $(n-2)$-tuple of $K_j\in {\cal K}_n$. If $A,B\in {\cal
K}_n$, we denote
$V(A,B,{\cal C})$ by $V(A,B)$. Then, for all $A,B,C\in {\cal K}_n$ we have
\begin{equation}V(A,A)V(B,C)\ls 2V(A,B)V(A,C).\end{equation}
\end{lemma}

\begin{proof}By the Aleksandrov-Fenchel inequality, for all $t,s\gr 0$ we
have
\begin{equation}V(B+tA,C+sA)^2-V(B+tA,B+tA)V(C+sA,C+sA)\gr 0\end{equation}
and
\begin{equation}V(sB+tC, A)^2-V(sB+tC,sB+tC)V(A, A)\gr 0.\end{equation}
Using the linearity of mixed volumes, from the first inequality we arrive at
\begin{align}
0 &\ls  g(t,s)+ t^2\left ( V(C,A)^2-V(A,A)V(C,C)\right )+s^2\left ( V(B,A)^2-V(A,A)V(B,B)\right )\\
\nonumber &+2ts\left (V(B,C)V(A,A)-V(B,A)V(C,A)\right ),
\end{align}
where $g$ is a linear function of $t$ and $s$. It follows that the quadratic term is
non-negative and hence, either $V(B,C)V(A,A)>V(B,A)V(C,A)$ or its discriminant
\begin{equation}
\left ( V(B,A)V(C,A)-V(B,C)V(A,A)\right )^2 - [V(B,A)^2-V(A,A)V(B,B)]\,[V(C,A)^2-V(A,A)V(C,C)]
\end{equation}
is non-positive. Working in the same way with the second inequality, we arrive at
\begin{align}
0 &\ls  t^2(V(C,A)^2-V(A,A)V(C,C))+s^2(V(B,A)^2-V(A,A)V(B,B))\\
\nonumber &+2ts(V(B,A)V(C,A)-V(B,C)V(A,A)).
\end{align}
This shows that if $V(B,C)V(A,A)>V(B,A)V(C,A)$ then the discriminant
of this second quadratic form (which is the same as before) is non-positive. It follows that, in both cases,
\begin{align}\left ( V(B,A)V(C,A)-V(B,C)V(A,A)\right )^2 &\ls [V(B,A)^2-V(A,A)V(B,B)]\,[V(C,A)^2-V(A,A)V(C,C)]\\
\nonumber &\ls V(B,A)^2V(C,A)^2.
\end{align}
Therefore,
\begin{equation}|V(B,A)V(C,A)-V(B,C)V(A,A)|\ls V(B,A)V(C,A),\end{equation}
and the lemma immediately follows.
\end{proof}

\medskip

We start with the proof of Theorem \ref{th:intro-3}. For any $u\in S^{n-1}$ we write $L_u$ for the line segment $[0,u]$.
Computing the volume of $K+tL_u$ we see that
\begin{equation}nV(K[n-1],L_u)=|P_E(K)|\end{equation} for every $K\in {\cal K}_n$, where $E=u^{\perp }$.
Linearity of mixed volumes shows that
\begin{equation}\label{eq:AF-1}nV(K_1,\ldots ,K_{n-1},L_u)=V_E(P_E(K_1),\ldots ,P_E(K_{n-1}))\end{equation}
for all $K_1,\ldots ,K_{n-1}\in {\mathcal K}_n$, where $V_E$ denotes mixed volume in $E$.
The next more general fact is due to Fedotov (see \cite{Burago-Zalgaller-book}).

\begin{lemma}\label{lem:AF-2}
Let $E\in G_{n,k}$ and $L_1,\ldots ,L_{n-k}$ be compact convex subsets of $E^{\perp }$.
If $K_1,\ldots ,K_k\in {\mathcal K}_n$, then
\begin{equation}{n\choose k}V(K_1,\ldots ,K_k,L_1,\ldots ,L_{n-k})=
V_{E}(P_E(K_1),\ldots ,P_E(K_k))V_{E^{\perp }}(L_1,\ldots ,L_{n-k}),\end{equation}
where $V_E,V_{E^{\perp }}$ denote mixed volumes on $E,E^{\perp }$ respectively.
\end{lemma}

\noindent {\bf Proof of Theorem \ref{th:intro-3}.} We apply Lemma \ref{lem:AF-1} with ${\cal C}=(K,\ldots ,K)$, $A=K$, $B=L_u=[0,u]$ and $C=L_v=[0,v]$.
We have
\begin{equation}V(L_u,L_v)V(K,K)\ls 2V(K,L_u)V(K,L_v).\end{equation}
Next, applying Lemma \ref{lem:AF-2} with ${\cal C}=(K,\ldots ,K)$, $L_1=[0,u]$, $L_2=[0,v]$ and $E={\rm span}\{ e_s:s\neq i,j\}$,
and observing that $V_{E^{\perp }}(L_u,L_v)=\frac{1}{2}\sqrt{1-\langle u,v\rangle ^2}$, we see that
\begin{equation}
V(L_u,L_v) = V(K,\ldots ,K,L_u,L_v)=\frac{1}{2}\sqrt{1-\langle u,v\rangle ^2}{n\choose 2}^{-1}|P_{u,v}(K)|.\end{equation}
Taking into account \eqref{eq:AF-1} and the fact that $V(K,K)=|K|$ we conclude that
\begin{equation}\frac{1}{n(n-1)}\sqrt{1-\langle u,v\rangle ^2}|P_{u,v}(K)|\,|K|\ls \frac{2}{n^2}|P_u(K)|\,|P_v(K)|,\end{equation}
and the result follows. $\hfill\Box $

\begin{remark}\rm Applying Lemma \ref{lem:AF-1} with ${\cal C}=(B_2^n,\ldots ,B_2^n)$, $A=B_2^n$, $B=K_1$
and $C=K_2$ we immediately see that for any pair of convex bodies $K_1,K_2$ in ${\mathbb R}^n$ we have
\begin{equation}V(B_2^n,B_2^n,{\cal C})V(K_1,K_2,{\cal C})\ls 2V(K_1,B_2^n,{\cal C})V(K_2,B_2^n,{\cal C}),\end{equation}
or equivalently,
\begin{equation}\label{eq:ostrover}|B_2^n|\,V(K_1,K_2,B_2^n[n-2])\ls 2V(K,B_2^n[n-1])V(K_2,B_2^n[n-1]).\end{equation}
It was mentioned in the introduction that this confirms the case $r=2$ of a conjecture of Hug and Schneider, up to a factor $2$. Recall that
\begin{equation}V(K_i,B_2^n[n-1])=\omega_n\int_{S^{n-1}}h_{K_i}(u)\,d\sigma (u)\end{equation}
for $i=1,2$ and that (see e.g. \cite{Schneider-book})
\begin{equation}V(K_1,K_2,B_2^n[n-2])=\omega_n\int_{S^{n-1}}h_{K_1}(u)\left (h_{K_2}(u)+\frac{1}{n-1}\Delta_Sh_{K_2}(u)\right )\,d\sigma (u)\end{equation}
where $\Delta_S$ is the spherical Laplace operator on $S^{n-1}$, therefore \eqref{eq:ostrover} implies that for any pair of support functions we have
\begin{equation}\int_{S^{n-1}}h_{K_1}(u)\left (h_{K_2}(u)+\frac{1}{n-1}\Delta_Sh_{K_2}(u)\right )\,d\sigma (u)
\ls 2\int_{S^{n-1}}h_{K_1}(u)\,d\sigma (u)\int_{S^{n-1}}h_{K_2}(u)\,d\sigma (u).\end{equation}
\end{remark}

\begin{remark}\rm The next result of Soprunov and Zvavitch (see \cite[Theorem 5.7]{Soprunov-Zvavitch-2016}) was mentioned in the
introduction. Let $A$ be any convex body in ${\mathbb R}^n$ and let $(K_1,\ldots ,K_r)$ be any $r$-tuple of convex bodies in ${\mathbb R}^n$.
Then,
\begin{equation}\label{eq:SZ}|A|^{r-1}V(K_1,\ldots ,K_r,A[n-r])\ls c_{n,r}\prod_{i=1}^rV(K_i,A[n-1]),\end{equation}
for some constant $c_{n,r}\ls n^rr^{r-1}$. Moreover, if $K_1,\ldots ,K_r$ are origin symmetric one can have the
same inequality with a constant $c_{n,r}^{\prime }\ls n^{r/2}r^{r-1}$. Applying
\eqref{eq:intro-14} with ${\cal C}=(A,\ldots ,A)$ and $B=K_1$, $C=K_2$ we immediately see that if $r=2$ then we get \eqref{eq:SZ} in the form
\begin{equation}\label{eq:SZ-2}|A|\,V(K_1,K_2,A[n-2])\ls 2\prod_{i=1}^2V(K_i,A[n-1]).\end{equation}
This is exactly the statement of Theorem \ref{th:intro-4}. A simple inductive argument shows that if $r=3$ then one can get \eqref{eq:SZ} in the form
\begin{equation}\label{eq:SZ-3}|A|^2\,V(K_1,K_2,K_3,A[n-3])\ls 8\prod_{i=1}^3V(K_i,A[n-1]).\end{equation}
More generally, for every $r\gr 2$ there exists $c_r>0$ (depending only on $r$) such that,
for any $n>r$ and any $r$-tuple $(K_1,\ldots ,K_r)$ of convex bodies in ${\mathbb R}^n$,
\begin{equation}\label{eq:SZ-r}|A|^{r-1}V(K_1,\ldots ,K_r,A[n-r])\ls c_r\prod_{i=1}^rV(K_i,A[n-1]).\end{equation}
Induction shows that \eqref{eq:SZ-r} holds true with $c_r\ls 2^{2^{r-1}-1}$.

Let us finally mention that Soprunov and Zvavitch have observed in \cite{Soprunov-Zvavitch-2016} that if $A=\Delta $ is an $n$-dimensional simplex
then \eqref{eq:SZ-r} holds true with constant $1$, and they conjecture that if a convex body $A$ in ${\mathbb R}^n$
satisfies \eqref{eq:SZ-r} with constant $1$ for all $r$ and all $K_1,\ldots ,K_r\in {\cal K}_n$ then $A$ must be an
$n$-dimensional simplex. In \cite{Saroglou-Soprunov-Zvavitch-2016} this conjecture is confirmed under the additional
hypothesis that $A$ is a polytope.
\end{remark}

\footnotesize
\bibliographystyle{amsplain}

\bigskip

\medskip

\thanks{\noindent {\bf Keywords:}  Convex bodies, volume of projections and sections, Loomis-Whitney inequality, uniform cover inequality.}

\smallskip

\thanks{\noindent {\bf 2010 MSC:} Primary 52A20; Secondary 52A23, 52A40, 46B06.}

\bigskip

\bigskip

\noindent \textsc{Silouanos \ Brazitikos}: Department of
Mathematics, National and Kapodistrian University of Athens, Panepistimiopolis 157-84,
Athens, Greece.

\smallskip

\noindent \textit{E-mail:} \texttt{silouanb@math.uoa.gr}

\bigskip

\noindent \textsc{Apostolos \ Giannopoulos}: Department of
Mathematics, National and Kapodistrian University of Athens, Panepistimiopolis 157-84,
Athens, Greece.

\smallskip

\noindent \textit{E-mail:} \texttt{apgiannop@math.uoa.gr}

\bigskip

\noindent \textsc{Dimitris-Marios \ Liakopoulos}: Department of
Mathematics, National and Kapodistrian University of Athens, Panepistimiopolis 157-84,
Athens, Greece.

\smallskip

\noindent \textit{E-mail:} \texttt{dimliako1@gmail.com}

\end{document}